\documentclass{amsart}
\usepackage[english]{babel}
\usepackage[utf8x]{inputenc}
\usepackage{listings}
\usepackage{caption}
\usepackage{float}
\usepackage{amsmath}
\usepackage{amsfonts} 
\numberwithin{equation}{section}
\numberwithin{figure}{section}
\numberwithin{table}{section}
\usepackage{mathrsfs}
\usepackage{amssymb}
\usepackage[colorinlistoftodos]{todonotes}
\usepackage{listings}
\usepackage{url}
\usepackage{graphicx}
\usepackage{bm}
\usepackage{subcaption}
\usepackage{subeqnarray}
\usepackage{multirow}
\newtheorem{theorem}{Theorem}[section]
\newtheorem{lemma}{Lemma}[section]

\graphicspath{ {./images/} }

\begin{document}
\title{Greedy Algorithm for Neural Networks for Indefinite Elliptic Problems}

\author{Qingguo Hong}

\address{Department of Mathematics and Statistics, Missouri University of Science and Technology\\
 Rolla, MO 65409, USA \\
qingguohong@mst.edu}

\author{Jiwei Jia}

\address{School of Mathematics, Jilin University\\
Changchun, Jilin 130012 \\
jiajiwei@jlu.edu.cn}

\author{Young Ju Lee}

\address{Department of Mathematics, Texas State University\\
San Marcos, TX 78666, USA \\
yjlee@txstate.edu}

\author{Ziqian Li}

\address{School of Mathematics, Jilin University \\
Changchun, Jilin 130012, China \\
zqli23@mails.jlu.edu.cn}

\maketitle

\begin{abstract}
The paper presents a priori error analysis of the shallow neural network approximation to the solution to the indefinite elliptic equation and and cutting-edge implementation of the Orthogonal Greedy Algorithm (OGA) tailored to overcome the challenges of indefinite elliptic problems, which is a domain where conventional approaches often struggle due to nontraditional difficulties due to the lack of coerciveness. A rigorous a priori error analysis that shows the neural network’s ability to approximate indefinite problems is confirmed numerically by  OGA methods. We also present a discretization error analysis of the relevant numerical quadrature. In particular, massive numerical implementations are conducted to justify the theory, some of which showcase the OGA’s superior performance in comparison to the traditional finite element method. This advancement illustrates the potential of neural networks enhanced by OGA to solve intricate computational problems more efficiently, thereby marking a significant leap forward in the application of machine learning techniques to mathematical problem-solving.
\end{abstract}

\section{Introduction}
Indefinite elliptic problems constitute a series of challenging open problems with profound implications across various fields of scientific computing and mathematical physics, including applications such as the Helmholtz equation, which describes wave propagation under certain conditions. The fundamental difficulty in addressing indefinite elliptic problems lies in the loss of coerciveness, meaning that standard energy methods and classical numerical theories that rely on positive definiteness are no longer applicable. This loss poses significant hurdles, both analytically and numerically, and demands innovative approaches for the development of effective and robust solution techniques. 

Over the years, various methods have been adopted to tackle indefinite elliptic problems. Notable among these are the multigrid methods, which have been explored for their efficiency in solving large-scale discretized problems, as outlined in \cite{bramble1988analysis} and \cite{wang1993convergence}. Domain decomposition methods offer another avenue to address these problems such as \cite{cai1992domain}. Furthermore, Krylov subspace methods, such as the Generalized Minimum Residual (GMRES) algorithm, have been enhanced for indefinite problems through preconditioning strategies, as exemplified in the work of \cite{xu1992preconditioned}. In recent years, there has been a surge in the application of finite element methods (FEMs) to indefinite elliptic problems. Mixed finite element methods (MFEM) are among the variants that have shown promising results in providing stable discretization and facilitating error analysis, as described in research papers like \cite{carstensen2021adaptive} and \cite{carstensen2022stability}. Another innovative approach is the virtual element method (VEM), a generalization of FEMs that allows for more flexibility in the mesh generation process and shows excellent potential for handling complex geometries, as discussed in \cite{carstensen2022priori}. Furthermore, Weak Galerkin (WG) finite element methods have been developed to provide remarkable error estimates and superconvergence properties, as highlighted in \cite{zhu2022superconvergent}.

Neural networks have undergone a renaissance in recent years, emerging as a formidable tool for tackling some of the most intractable problems in computational science. Among these challenges, solving partial differential equations (PDEs) stands out as an area where neural networks are having a particularly significant impact. Recognized for their remarkable ability to model and approximate functions, neural networks offer an alternative to traditional numerical methods, especially for complex, high-dimensional, and nonlinear PDEs, which are typically infeasible for grid-based techniques due to the curse of dimensionality.
One of the groundbreaking advances that encapsulates the use of neural networks in PDEs is the Physics-Informed Neural Networks (PINNs) framework, as introduced in \cite{raissi2019physics}. PINNs have transformed the landscape by integrating physical laws directly into the learning process, which traditionally relied only on data.  Applications have been diverse and successful, ranging from fractional equations as seen in \cite{pang2019fpinns}, to complex fluid dynamics problems tackled in \cite{jin2021nsfnets} and \cite{bihlo2022physics}, and even to the domain of phase field models, as researched in \cite{mattey2022novel}. Adjacent to PINNs are various other techniques that leverage neural networks to solve PDEs. The Deep Ritz Method \cite{yu2018deep} is an illustrative example that adapts the variational approach of the Ritz method for neural networks. Weak Adversarial Networks, outlined in \cite{zang2020weak}, bring the power of generative adversarial networks to solve PDEs. Moreover, the development of network architectures such as Deeponet \cite{lu2021learning} and \cite{li2020fourier} opened avenues for creating neural networks capable of learning mappings from functional data to PDE solutions, further exemplifying the breadth of neural network applications in this field.

While neural networks have shown remarkable capabilities in numerically tackling PDEs, the theoretical underpinnings of these methods, especially as they relate to error and convergence analysis, have lagged behind their practical applications. Deep neural networks, in particular, pose a challenge for rigorous analysis due to their complex, highly nonlinear structures and large parameter spaces. Despite these challenges, there is an emerging body of work dedicated to developing a deeper understanding of neural network-based algorithms' theoretical aspects. A notable contribution in this domain is the finite neuron method (FNM), proposed and analyzed in \cite{CiCP-28-1707}. The FNM presents an innovative approach where the architecture of the network is restricted to a finite number of neurons, rendering the system more amenable to classic notions of stability and convergence. This conceptual framing allows for a rigorous error and convergence analysis to be conducted within this finite-dimensional setting.
Further reinforcing this theoretical framework, there have been enlightening studies such as \cite{siegel2022high} and \cite{siegel2023characterization}, which delve into the approximation theory of shallow neural networks. Additionally, \cite{hong2021priori} contributes a priori analysis of shallow neural networks solving PDEs. Building upon the foundations laid by the FNM, Xu et al introduced advanced algorithmic strategies with a theoretical backing, namely the relaxed greedy algorithm (RGA) and orthogonal greedy algorithm (OGA) as discussed in \cite{siegel2023greedy}. These techniques further optimize the training of shallow neural networks by systematically enhancing the process of selecting and combining neuron functions to minimize residuals. This algorithm gives complete error and convergence analysis and shows great potential in solving PDEs.

In this paper, by adapting the classical theory due to Schatz and Wang \cite{schatz1996some} for the FNM, we establish the optimal convergence estimate for the neural networks to the solution to the indefinite problem. This estimate is then shown to be achieved by the application of the orthogonal greedy algorithm numerically via extensive numerical experiments. The outline of the remainder of the paper is as follows. In Section 2, we describe preliminaries and the OGA method to solve the problem. In Section 3, we introduce the error analysis and convergence analysis. In Section 4, the results of numerical experiments illustrating the earlier derived theory are given. Finally, a conclusion is given in Section 5.

\section{The OGA Method for Shallow Neural Networks}
\subsection{Preliminaries}
In this paper, we use the standard Sobolev space notations. Given a non-negative integer $k$ and a bounded domain $\Omega\subset \mathbb{R}^d$, let
\begin{equation}
	H^k(\Omega):=\{v\in L^2(\Omega),\partial^\alpha v\in L^2(\Omega), \vert \alpha\vert\leq k\}
\end{equation}
be standard Sobolev spaces with $\|\cdot\|_k$ norm and $|\cdot|_{k}$ seminorm, given respectively as follows: 
$$\Vert v\Vert_{k}:=\left( \sum_{\vert \alpha \vert\leq k}\Vert \partial^\alpha v\Vert_0^2 \right)^{1/2} \quad \mbox{ and } \quad \vert v\vert_{k}:=\left( \sum_{\vert \alpha \vert= k}\Vert \partial^\alpha v\Vert_0^2 \right)^{1/2}.$$
For $k=0$, $H^0(\Omega)$ is the standard $L^2(\Omega)$ space with the inner product denoted by $(\cdot,\cdot)$.

We consider the following simplified second-order elliptic equation with certain boundary conditions:
\begin{equation}
\label{eq:indefinite}
		\left\{
			\begin{aligned}
				 -\nabla \cdot (A\nabla u)+ cu&=f\quad \rm{in} \ \Omega, \\
				 B^k(u)&=0 \quad \rm{on}\ \partial\Omega,
			\end{aligned}
		\right.
	\end{equation}
where $B^k(u)$ denotes the Dirichlet, Neumann, or mixed boundary conditions. For simplicity, we use the Neumann boundary condition in this paper. Here, $A=(a_{ij}(x))_{d\times d}$ is a symmetric, positive definite matrix of coefficients. The weak formulation \eqref{eq:indefinite} seeks $u\in H^1(\Omega)$ such that  
\begin{equation}
\label{eq:weak}
a(u,v) := (A\nabla u,\nabla v)+(cu,v)=(f,v),\quad v\in H^1(\Omega).
\end{equation}
We assume that $A$ is uniformly elliptic, i.e, there is a positive constant $\beta$ such that
\begin{equation}
\label{eq:positive}
	\xi^T A\xi \geq \beta \xi^T\xi,\quad \forall\xi\in \Omega.
\end{equation}
On the other hand, we shall assume that the coefficient $c$ in (\ref{eq:indefinite}) is negative. Thus, the bilinear form $a(\cdot,\cdot)$ is indefinite. Namely, it does not satisfy the coerciveness. 

\subsection{Shallow Neural Network}
We introduce a set of functions by finite expansions concerning a dictionary $\mathbb{D}\subset H^m(\Omega)$. The set is defined as
\begin{equation}
	\sum\nolimits_{n,M}(\mathbb{D})=\left\{ \sum_{i=1}^n a_ig_i,\, g_i\in\mathbb{D},\, \sum_{i=1}^n\vert a_i\vert\leq M \right\}. 
\end{equation}
Note that here we restrict the $\ell^1$ norm of the coefficients $a_i$ in the expansion. For shallow neural networks with $\rm{ReLU}^k$ activation function $\sigma=\max\{0,x\}^k$, the dictionary $\mathbb{D}$ of $d$-dimension would be taken as 
\begin{equation}
	\mathbb{D}=\mathbb{P}_k^d=\left\{\sigma_k({\bm{\omega}}\cdot{\bm{x}}+b):\, {\bm{\omega}}\in S^{d-1},\, b\in [c_1,c_2]\right\}\subset L^2(B_1^d),
\end{equation}
where $S^{d-1}=\{{\bm{\omega}}\in \mathbb{R}^d:\,\vert{\bm{\omega}}\vert=1 \}$ is the unit sphere and $B_1^d$ is the closed $d-$dimensional unit ball. Here $c_1$ and $c_2$ are chosen to satisfy
\begin{equation}
	c_1<\inf\{{\bm{\omega}}\cdot{\bm{x}}:{\bm{x}}\in \Omega^d,{\bm{\omega}}\in S^{d-1}\}<\sup\{{\bm{\omega}}\cdot{\bm{x}}:{\bm{x}}\in \Omega^d,{\bm{\omega}}\in S^{d-1}\}<c_2.
\end{equation}
Given a dictionary $\mathbb{D}$, the space we measure regularity of the solution is the $\mathcal{K}_1(\mathbb{D})$ space \cite{devore1998nonlinear}. The closed convex hull of $\mathbb{D}$ is defined by
\begin{equation}
	B_1(\mathbb{D})=\overline{\bigcup_{n=1}^\infty \sum\nolimits_{n,1}(\mathbb{D})},
\end{equation}
and the $\mathcal{K}_1(\mathbb{D})$ norm is defined by
\begin{equation}
	\Vert f\Vert_{\mathcal{K}_1(\mathbb{D})}=\inf\{t>0: f\in tB_1(\mathbb{D}) \}.
\end{equation}
So the $\mathcal{K}_1(\mathbb{D})$ space is given by
\begin{equation}
	\mathcal{K}_1(\mathbb{D}):=\{f\in H^m(\Omega): \Vert f\Vert_{\mathcal{K}_1(\mathbb{D})}<\infty\}.
\end{equation}
It should be noted that if $\sup_{g\in \mathbb{D}}\Vert g\Vert_H<\infty$ where $\Vert \cdot\Vert_H$ is the norm induced by the energy inner product, the $\mathcal{K}_1(\mathbb{D})$ space is a Banach space \cite{siegel2023characterization}. Further, we define the $B_M(\mathbb{D})$
\begin{equation}
	B_M(\mathbb{D})=\{u\in\mathcal{K}_1(\mathbb{D}): \Vert u\Vert_{\mathcal{K}_1(\mathbb{D})}\leq M \}
\end{equation}
to analyze the quadrature error in Section \ref{sec:discretization}.

\subsection{The OGA Method}
The orthogonal greedy algorithm is an algorithm used for approximating functional representations in a given function space. It is particularly relevant in the context of sparse approximations and has applications in signal processing, numerical analysis, and machine learning. The algorithm iteratively builds an approximation of a target function by selecting elements from a dictionary that are not necessarily orthogonal. The mathematical framework set out by DeVore et al. DeVore et al. \cite{devore1996some} investigates a quantifiable analysis of the OGA's performance. Given a dictionary $\mathbb{D}$ within a Hilbert space $H$, and a target function $f \in H$, the OGA constructs an approximation $f_n$ by iteratively selecting elements $(\phi_{1}, \phi_{2}, \ldots, \phi_{n})$ from $\mathbb{D}$. These elements are chosen based on the criterion of maximal inner product with the current residual, which after the $k$-th step is defined as the remainder $r_k = f - f_k$. Precisely, the element $\phi_{k+1}$ selected in the $(k+1)$-th step satisfies:
\begin{equation}
\phi_{k+1} = \arg\max_{\phi_{k+1} \in \mathbb{D}} |\langle r_k, \phi_{k+1} \rangle_H|, 
\end{equation}
where $\langle \cdot,\cdot \rangle_H$ is an inner product on $H$. Subsequently, the function $f$ is orthogonally projected onto the subspace spanned by the chosen elements, updating the residual in the process. DeVore et al stipulated bounds on the error of approximation $\Vert f - f_n\Vert_H$, expressed in terms of the best $n$-term approximation from $\mathbb{D}$, denoted by $\sigma_n(f, \mathbb{D})$. They establish that under appropriate conditions related to the coherence properties of the dictionary, the OGA maintains an approximation quality comparable to that of the optimally selected $n$-term representation.

Assuming $c$ in (\ref{eq:indefinite}) is not an eigenvalue, the equation \eqref{eq:indefinite} admits a unique solution. However, $\langle \cdot,\cdot \rangle_H$ induced from the equation will not be an inner product. Thus, it does not induce a norm $\|\cdot\|_H$ either. Our numerical tests show that still OGA works. More precisely, let the $\mathcal{K}_1(\mathbb{D})$ space be the target space in the analysis of the greedy algorithm \cite{temlyakov2008greedy}. Given the dictionary $\mathbb{D}$ and a target function $u$, greedy algorithms approximate the target function by a finite linear combination of dictionary elements:
\begin{equation}
	u_n=\sum_{i=1}^n a_ig_i,\quad g_i\in \mathbb{D}.
\end{equation}
The OGA method is given by
\begin{equation}
\label{eq:OGA}
	u_0=0,\,g_n=\arg\max_{g\in \mathbb{D}}\vert \langle g,u_{n-1}-u\rangle_H\vert,\,u_n=P_n(u),
\end{equation}
where $P_n$ is a projection onto the span of $g_1,\cdots,g_n$. Note that $\langle \cdot,\cdot\rangle_H$ is induced by the weak form given in (\ref{eq:weak}), i.e.
\begin{equation}
\label{eq:inner}
\langle u, v\rangle_H = a(u,v) := (A\nabla u, \nabla v)+(cu,v), \quad \forall u, v.
\end{equation}
Again, it is not an inner product since it is indefinite. In whatsoever, we see that the numerical algorithm is well-defined. We note that 
\begin{align}\label{eq:product}
	\left\langle g, u_{n-1}-u\right\rangle_H &= (\nabla g,A\nabla(u_{n-1}-u))+(g, cu_{n-1}-cu)\notag\\
	&= (\nabla g,A\nabla u_{n-1})+(g,cu_{n-1})-(\nabla g,A\nabla u)-(g,cu) \notag\\
	&= (\nabla g,A\nabla u_{n-1})+(g,cu_{n-1})-(f,g) \notag\\
	&= \int_{\Omega}(\nabla g\cdot A\nabla u_{n-1}+cgu_{n-1}-fg){\rm{d}}\bm{x}.
\end{align}
According to (\ref{eq:product}), solving the argmax problem only depends on the numerical solution of the previous step, and the basis function $g_n$ on the $n$-th step is determined by the $g$ satisfying the argmax problem. To solve the argmax problem, we consider the following equivalent optimization problem:
\begin{equation}\label{eq:optimization_min}
	g_n=\arg\min_{g\in \mathbb{D}}-\frac{1}{2}\langle g,u_{n-1}-u\rangle_H^2.
\end{equation}
Here we choose the shallow neural network dictionary $\mathbb{D}\subset\mathbb{R}^d$ as $\mathbb{D}=\mathbb{P}_k^d$, which is parameterized by $\bm{\omega}\in S^{d-1}$ and $b\in [-c,c]$. The optimization problem (\ref{eq:optimization_min}) is non-convex so it would be difficult to obtain the global minimum. The approach is to obtain a suitable initial guess by choosing many initial samples on the $\bm{\omega}-b$ parameter and evaluating the function (\ref{eq:optimization_min}) to find the smallest one. After the initial guess is chosen, we use Newton's or gradient descent method for further optimization. 

The key to building up the dictionary $\mathbb{D}$ is to find the best initial samples at $(\bm{\omega}_i, b_j)$. For parameter $\bm{\omega}_i$, in order to enforce the constraint $\Vert\bm{\omega}_i\Vert=1$, we take $\omega_i= 1$ or $-1$ for 1D case. For 2D case, \cite{siegel2023greedy} proposed $\bm{\omega}_i=(\cos\theta_i,\sin\theta_i), \theta_i=\frac{2\pi i}{N_\theta}, (i=0,1,\cdots,N_\theta)$ where $N_\theta$ is the number of sampling point of $\theta$ with high computational cost. Instead, we choose $\bm{\omega}_i=(\pm 1, \pm 1)$ which works much faster with higher accuracy. This strategy decreases the computational complexity from $O(N_\theta N_b)$ to $O(4N_b)$. Similarly, $\bm{\omega}_i=(\pm 1, \pm 1, \pm 1)$ for 3D case. For parameter $b_j$, we choose $b_j=-c+\frac{2cj}{N_b}\,(j=0,1,\cdots,N_b)$, where $N_b$ is the number of sampling point of $b_j$. 

For the projection step, due to (\ref{eq:weak}) and
\begin{equation}
	u_n=P_n(u)=\sum_{i=1}^n a_ig_i,
\end{equation}
where $a_i$ are obtained by solving the following linear problem: 
\begin{equation}
\label{eq:linear}
\int_{\Omega} A\nabla \left ( \sum_{i=1}^n a_i g_i \right )\cdot\nabla g_j {\rm{d}}\bm{x} + \int_{\Omega} c \left ( \sum_{i=1}^na_i g_i \right )g_j {\rm{d}}\bm{x} = \int_{\Omega} fg_j{\rm{d}}\bm{x}, \quad \forall j =1,\cdots,n. 
\end{equation}

\section{A Priori Error Estimate and Convergence Analysis}
In this section, we first discuss the error analysis of the neural network approximation to the indefinite problem. Then we analyze different components separately. The variational formulation of (\ref{eq:indefinite}) is equivalent to find $u\in H^1(\Omega)$ such that
\begin{equation}
u = \arg\min_{v\in H^1(\Omega)}\mathcal{R}(v),
\end{equation}
where the energy functional $\mathcal{R}(v)$ is given as follows:
\begin{equation}
\mathcal{R}(v) = \frac{1}{2} a(u,u) - (f,u).
\end{equation}
We remark $\mathcal{F}_{\Theta}$ the function class based on the shallow neural network dictionary $\mathbb{D}$. The corresponding function space is $V_n$ where $n$ is the number of neurons. The goal is to optimize the bounded generalization error:
\begin{equation}
\label{eq:generalerror}
	\mathcal{R}(u_{\Theta,N})-\mathcal{R}(u)= \mathcal{R}(u_{\Theta})-\mathcal{R}(u)+ \mathcal{R}(u_{\Theta,N})-\mathcal{R}(u_{\Theta}),
\end{equation}
where $u_{\Theta}$ is the minimizer of the true risk over the function class $\mathcal{F}_{\Theta}$, $u$ is the global minimizer of the risk and $u_{\Theta,N}$ is $u_{\Theta}$ calculated by Gauss quadrature. Through (\ref{eq:generalerror}), the priori error of the method is divided into two parts:
\begin{itemize}
	\item $\mathcal{R}(u_{\Theta})-\mathcal{R}(u)$ is the modelling error. It measures the approximation ability of the model class $\mathcal{F}_{\Theta}$ to capture the true solution $u$.
	\item $\mathcal{R}(u_{\Theta,N})-\mathcal{R}(u_{\Theta})$ is the discretization error. It measures the error incurred by the numerical quadrature. 
\end{itemize}
We analyze the different error terms separately in the following subsections.

\subsection{Modelling Error}\label{sec:modeling}
Through the bilinear form defined in (\ref{eq:inner}), the pseudo-norm $\Vert \cdot \Vert_H^2$ is induced as
\begin{equation}
\Vert u \Vert_H^2 := a(u,u). 
\end{equation}
Thus, the modeling error is transformed to 
\begin{equation}
	\mathcal{R}(u_{\Theta})-\mathcal{R}(u) = \frac{1}{2}\Vert u-u_{\Theta} \Vert_H^2=a(u-u_{\Theta},u-u_{\Theta}),
\end{equation}
where $a(u,v)$ is the symmetric bilinear form of (\ref{eq:indefinite}). It should be noted that $a(u,v)$ is bounded, i.e., there exists $c > 0$ such that 
\begin{equation}\label{cont} 
|a(u,v)| \leq c \Vert u\Vert _1\Vert v\Vert _1 \quad \forall u, v \in H^1(\Omega).     
\end{equation}
The crucial point of the indefinite problem is the way to handle the lack of coerciveness. First, we introduce the $\rm{G\mathring{a}rding's\ inequality}$:
\begin{lemma}[$\rm{G\mathring{a}rding's\ inequality}$]
\label{lm:garding}
	Suppose (\ref{eq:positive}) holds, there exists a constant $G\geq 0$, such that
	\begin{equation}
	\label{eq:gardinglemma}
		a(u,u) \geq \beta \Vert u\Vert_1^2 - G\Vert u\Vert_0^2, \quad \forall u\in H^1(\Omega). 
	\end{equation}
\end{lemma}
\begin{proof}
	\begin{align}
	\label{eq:garding}
		a(u,u)+G\Vert u\Vert_0^2 &\geq \beta \Vert \nabla u\Vert^2 + c\Vert u\Vert^2 + G\Vert u\Vert_0^2 \notag\\
		&\geq \beta \Vert \nabla u\Vert^2 + \gamma\Vert u\Vert^2 + G\Vert u\Vert_0^2,
	\end{align}
where 
$$\gamma:={\rm{ess}} \inf\{c(x):x\in \Omega\}.$$
Since $c(x)$ is negative in indefinite problems, $\gamma$ is also negative. By choosing $G\geq \beta-\gamma>0$ combined with (\ref{eq:garding}),
\begin{equation}
	a(u,u)+G\Vert u\Vert_0^2 \geq \beta \Vert u\Vert_1^2.
\end{equation}
The proof is completed.
\end{proof}
According to Lemma \ref{lm:garding}, we have
\begin{equation}
\label{eq:errorbound}
	a(u-u_{\Theta},u-u_{\Theta}) \geq \beta \Vert u-u_{\Theta}\Vert_1^2 - G\Vert u-u_{\Theta}\Vert_0^2.
\end{equation} 
If the $L^2$ error can be bounded by the $H^1$ error similar to the Aubin-Nitsche trick in finite element theory, the $H^1$ error can be controlled by the modeling error. Further, it is bounded through (\ref{cont}). We shall need a couple of lemmas to establish the theorem. The first lemma \cite{CiCP-28-1707} is given as follows: 
\begin{lemma}
\label{lm:neuron}
For $\forall u \in H^1(\Omega)$ and $\varepsilon > 0$, we can find $\chi \in V_n$ such that the number of neurons $n$ depends on $u$ and $\varepsilon$, 
\begin{equation}
\Vert u - \chi\Vert _1 \leq \varepsilon. 
\end{equation}
\end{lemma}
We begin with the compactness result \cite{schatz1996some} for the solution space in $H^1(\Omega)$. 
\begin{lemma}
\label{lm:precompact}
Let $D = \{f : f \in L^2(\Omega), \Vert f\Vert _0 = 1\}$ be the unit sphere in $L^2(\Omega)$. Let $W = \{u: u = Tf, f \in D\}$ where $u = Tf \in H^1(\Omega)$, such that 
\begin{equation}
a(Tf,v) = a(u,v) = (f,v), \quad \forall v \in H^1(\Omega). 
\end{equation}
Then $W$ is precompact in $H^1(\Omega)$. 
\end{lemma}
\begin{proof}
The proof follows from the following inequality:
\begin{equation}
\Vert u\Vert _1 = \Vert Tf\Vert _1 \leq c_4 \Vert f\Vert _{-1}, \quad \forall f \in D. 
\end{equation}
This means, $T$ is continuous from $H^{-1}(\Omega)$ to $H^1(\Omega)$. Thus, $W$ is precompact in $H^1(\Omega)$. 
This completes the proof. 
\end{proof}
We now establish Lemma \ref{lm:compact} that compact subsets $V\in H^1(\Omega)$ can be uniformly approximated by elements in $V_n\in H^1(\Omega)$. 
\begin{lemma}
\label{lm:compact}
Let $V$ be a compact subset of $H^1(\Omega)$. Then given any $\varepsilon > 0$, there exists an $n_0(\varepsilon, V)$ such that for $\forall v \in V$ and $\forall n > n_0$, there exists a $\chi \in V_n$ such that 
\begin{equation}
\Vert v - \chi\Vert_1 \leq \varepsilon. 
\end{equation}
\end{lemma}
\begin{proof}
Since $V$ is a compact, for given $\varepsilon/2$, there exist $v_1, \cdots, v_M \in H^1(\Omega)$ such that 
\begin{equation}
V \subset \bigcup_{i=1}^M \rho(\frac{\varepsilon}{2};v_i), 
\end{equation}
where $\rho(\varepsilon/2;v_i)$ is the ball of radius $\varepsilon/2$ centered at $v_i$ in $H^1(\Omega)$. Then, according to Lemma \ref{lm:neuron}, there exists $n_i$ such that $\chi_i \in V_{n_i}$ such that 
\begin{equation}
\Vert v_i - \chi_i\Vert_1 \leq \varepsilon/2,  
\end{equation}
Now, remark $n_0 = \max \{n_1, \cdots, n_M\}$, for $v \in V$, we have 
\begin{equation}
\Vert v - \chi_j\Vert_1 \leq \Vert v - v_j\Vert _1 + \Vert v_j - \chi_j\Vert _1 \leq \varepsilon. 
\end{equation}
This completes the proof. 
\end{proof}

The finite neuron method is to find $u_{\Theta} \in V_n$ such that 
\begin{equation}\label{dis}
a(u_{\Theta}, \phi) = (f,\phi) \quad \forall \phi \in V_n. 
\end{equation}
We note that the quasi-orthogonality holds: 
\begin{equation}\label{or}
a(u-u_{\Theta}, \phi) = 0, \quad \forall \phi \in V_n. 
\end{equation}
The main theorem is given as follows:
\begin{theorem}
\label{thm:thick}
Suppose that $a(\cdot,\cdot)$ satisfies the inequalities \ref{cont} and \ref{eq:gardinglemma} while $u \in H^1(\Omega)$ and $u_{\Theta} \in V_n$ satisfy \ref{dis} and \ref{or}. Then, given $\forall \varepsilon > 0$, there exists an $n_0 = n_0(\varepsilon) > 0$ such that for $\forall n > n_0(\varepsilon)$, we have 
\begin{equation}
\Vert u - u_{\Theta}\Vert _0 \leq \varepsilon \Vert u-u_{\Theta}\Vert _1. 
\end{equation}
\end{theorem}
\begin{proof}
We consider the adjoint bilinear form to be $a(u,v)$ defined by 
\begin{equation}
a^*(u,v) = a(v,u), \quad \forall u,v \in H^1(\Omega). 
\end{equation}
The adjoint problem has a unique solution in $H^1(\Omega)$ for $\forall g \in H^{-1}(\Omega)$ such that 
\begin{equation}
a^*(w^*,v) = (g,v), \quad \forall v \in H^1(\Omega)     
\end{equation}
and 
\begin{equation}
\Vert w^*\Vert _1 \leq c\Vert g\Vert _{-1}.     
\end{equation}
We consider the solution set for a given $g$ in the unit sphere in $L^2(\Omega)$. Namely, 
\begin{equation}
W^* = \{ w^* : a^*(w^*,v) = (g,v) \quad \forall v \in H^1(\Omega) \mbox{ with } \Vert g\Vert _0=1\}.     
\end{equation}
Then, $W^*$ is precompact according to Lemma \ref{lm:precompact}. We now invoke the identity: 
\begin{equation}
\Vert u - u_{\Theta}\Vert _0 = \sup_{g \in L^2(\Omega), \Vert g\Vert _0 = 1} (u - u_{\Theta}, g). 
\end{equation}
Furthermore, due to (\ref{cont}), for $\forall \chi \in V_n$, we have 
\begin{equation}
(u - u_{\Theta}, g) = a^*(w^*, u - u_{\Theta}) = a(u - u_{\Theta}, w^* - \chi) \leq c \Vert u - u_{\Theta}\Vert _1 \Vert w^* - \chi\Vert _1.     
\end{equation}
Combined with Lemma \ref{lm:compact}, this shows that 
\begin{equation}
\Vert u - u_{\Theta}\Vert_0 \leq \varepsilon \Vert u - u_{\Theta}\Vert_1.     
\end{equation}
This completes the proof. 
\end{proof}
Combined with (\ref{eq:errorbound}) and Theorem \ref{thm:thick}, we have 
\begin{equation}
	\Vert u - u_{\Theta}\Vert_1^2 \leq \frac{1}{\beta-G\varepsilon}a(u-u_{\Theta},u-u_{\Theta})
\end{equation}
with the bounded bilinear form due to (\ref{cont}). This in turns leads to the quasi-optimality given as follows:
\begin{equation}
\|u - u_{\Theta}\|_1 \lesssim \inf_{\chi \in V_n} \|u - \chi\|_1. 
\end{equation}
Therefore, the $L^2$ and $H^1$ errors can be controlled optimally. 

\subsection{Discretization Error}\label{sec:discretization}
In this section, we introduce the discretization error incurred by the Gauss-Legendre quadrature. Let $\mathcal{T}_h \subset \Omega$ be a partition on $\Omega$ with mesh size $h$, where $h=\mathcal{O}(N^{-\frac{1}{d}})$ and $N$ is the number of quadrature points. For each $T_l\in \mathcal{T}_h$, $l=1,\cdots,L$, the quadrature rule satisfies
\begin{equation}
    \int_{T_l}p(x)dx=\sum_{i=1}^t p(x_{l,i})\omega_i, \quad \forall p\in\mathcal{P}_{2t-1}(T_l),
\end{equation}
where $\mathcal{P}_{2t-1}(T)$ is the space of polynomials with degree less equal than $2t-1$. Therefore, we have
\begin{equation}
\label{eq:quad}
    \int_{\Omega}p(x)dx=\sum_{l=1}^L\int_{T_l}p(x)dx=\sum_{l=1}^L\sum_{i=1}^t p(x_{l,i})\omega_i=\sum_{j=1}^Np(x_j)\omega_j,\quad \forall p\in \mathcal{P}_{2t-1}(\mathcal{T}_h),
\end{equation}
where $\mathcal{P}_{2t-1}(\mathcal{T}_h)=\{g\in L^2(\Omega):g\vert_T\in \mathcal{P}_{2t-1}(T),\forall T\in \mathcal{T}_h \}$ is the space of piece-wise polynomial functions on the partition $\mathcal{T}_h$. The estimate of the discretization error is under the following assumptions:
\begin{itemize}
	\item the coefficients and $f$ in (\ref{eq:indefinite}) satisfy $\Vert A\Vert_{W^{r,\infty}(\Omega)}$, $\Vert c\Vert_{W^{r,\infty}(\Omega)}$, $\Vert f\Vert_{W^{r,\infty}(\Omega)}\leq K_r$,
	\item the dictionary $\mathbb{D}$ satisfies $\sup\limits_{d\in\mathbb{D}}\Vert d\Vert_{W^{r+2,\infty}(\Omega)}=C<\infty$,
	\item the discretization solution $u\in \mathcal{K}_1(\mathbb{D})$ with $\Vert u \Vert_{ \mathcal{K}_1(\mathbb{D})}\leq M$, i.e. $u\in B_M(\mathbb{D})$.
\end{itemize} 
Since 
\begin{equation}
	p(x)=\frac{1}{2}A(\nabla u(x))^2+\frac{1}{2}c(u(x))^2-f(x)u(x), \quad \forall u\in W^{r+2}(\Omega),
\end{equation}
we have
\begin{equation}
	\left\vert \mathcal{R}(u_{\Theta,N})-\mathcal{R}(u_{\Theta})\right\vert = \left\vert \int_{\Omega}p(x)dx-\sum_{j=1}^N p(x_j)\omega_j\right\vert.
\end{equation}
Using the Bramble-Hilbert Lemma \cite{bramble1970estimation}, we get
\begin{equation}
	\left\vert \int_{\Omega}p(x)dx-\sum_{j=1}^N p(x_j)\omega_j\right\vert \leq C_mN^{-\frac{r+1}{d}}\Vert p\Vert_{r+1,\infty}, \quad \forall p\in W^{r+1,\infty}(\Omega).
\end{equation}
Due to the coefficients and $f$ are bounded in the assumptions, $p(x)$ is bounded by
\begin{equation}
	\Vert p\Vert_{r+1,\infty}\leq C_r K_r \Vert u\Vert_{r+2,\infty}^2.
\end{equation}
Therefore,
\begin{equation}
	\left\vert \int_{\Omega}p(x)dx-\sum_{j=1}^N p(x_j)\omega_j\right\vert \leq C_m C_r K_r N^{-\frac{r+1}{d}} \Vert u\Vert_{r+2,\infty}^2.
\end{equation}
Since we have $u\in B_M(\mathbb{D})$, thus $\Vert u\Vert_{r+2,\infty}\leq CM$. It indicates that
\begin{equation}
	\left\vert \int_{\Omega}p(x)dx-\sum_{j=1}^N p(x_j)\omega_j\right\vert \leq C_m C_r C^2 K_r N^{-\frac{r+1}{d}} M^2,
\end{equation}
i.e. the discretization error is bounded.

\section{Numerical Experiments}
In this section, we provide numerical experiments demonstrating the effectiveness of the greedy algorithm in solving indefinite elliptic problems. For all the experiments below, the numerical quadrature uses Gaussian quadrature (\ref{eq:quad}) with the default setting $t=2$ and $L=4000$ for 1D, $t=2\times2$, $L=400\times400$ for 2D and $t=2\times2\times2$, $L=50\times50\times50$ for 3D. We use $u_h$ to denote the numerical solution and use $u$ to represent the analytical solution. We calculate the $L^2$ and $H^1$ error and their corresponding convergence orders. The numerical experiments are performed on a workstation with 2 24 Core Intel Xeon Platinum 8269CY CPU, 512GB RAM, and a Ubuntu 20.04 operating system that implements MATLAB.

\subsection{Example 1}
We consider the 1D elliptic equation
\begin{equation}
		\left\{
			\begin{aligned}
				 &-u''+cu=f,\quad x\in (-1,1), \\
				 &u'(-1)=u(1)=0,
			\end{aligned}
		\right.
	\end{equation}
with the analytical solution $u(x)=\cos(\pi x)$. We apply the orthogonal greedy algorithm (OGA) with dictionary $\mathbb{P}_2^1$ (i.e. corresponding to ${\rm{ReLU}}^2$). For the parameters of the dictionary, $\omega=\pm 1$ and $b\in [-2,2]$. We present $c=-1$ and $-1e+06$ respectively and Table \ref{table:1D_nega_c1}-\ref{table:1D_nega_c4} show the results. The results show that the convergence orders are constant with the increase of the number of neurons $n$.

\begin{table}[H]
\begin{center}
{{\begin{tabular}{|c|c|c|c|c|}
\cline{1-5}
$n$ & $\Vert u-u_h\Vert _{L^2}$ & order & $\Vert u-u_h\Vert _{H^1}$ & order \\ \hline 
\cline{1-5} 
16 & 6.740e-04 & - & 2.428e-02 & - \\ \hline 
32 & 6.793e-05 & 3.31 & 5.290e-03 & 2.20 \\ \hline 
64 & 7.699e-06 & 3.14 & 1.232e-03 & 2.10 \\ \hline 
128 & 8.990e-07 & 3.10 & 2.987e-04 & 2.04 \\ \hline 
256 & 1.117e-07 & 3.01 & 7.520e-05 & 1.99 \\ \hline
\end{tabular}
\caption{$L^2$ error, $H^1$ error and order for greedy algorithm solving 1D elliptic equation when $c=-1$}
\label{table:1D_nega_c1}
}}\end{center}
\end{table}

\begin{table}[H]
\begin{center}
{{\begin{tabular}{|c|c|c|c|c|}
\cline{1-5}
$n$ & $\Vert u-u_h\Vert _{L^2}$ & order & $\Vert u-u_h\Vert _{H^1}$ & order \\ \hline 
\cline{1-5} 
16 & 5.352e-04 & - & 2.296e-02 & - \\ \hline 
32 & 6.746e-05 & 2.99 & 5.518e-03 & 2.06 \\ \hline 
64 & 8.110e-06 & 3.06 & 1.353e-03 & 2.03 \\ \hline 
128 & 1.014e-06 & 3.00 & 3.352e-04 & 2.01 \\ \hline 
256 & 1.301e-07 & 2.96 & 8.568e-05 & 1.97 \\ \hline
\end{tabular}
\caption{$L^2$ error, $H^1$ error and order for greedy algorithm solving 1D elliptic equation when $c=-1e+06$}
\label{table:1D_nega_c4}
}}\end{center}
\end{table}

\subsection{Example 2}
We consider the 2D elliptic equation
\begin{equation}
		\left\{
			\begin{aligned}
				 &-\Delta u+cu=f,\quad x\in (0,1)^2, \\
				 &\frac{\partial u}{\partial n}=0, \quad x\in \partial(0,1)^2,
			\end{aligned}
		\right.
	\end{equation}
with the analytical solution $u(x,y)=\cos(10\pi x)\cos(10\pi y)$. We apply the orthogonal greedy algorithm (OGA) with dictionary $\mathbb{P}_2^2$ (i.e. corresponding to ${\rm{ReLU}}^2$). For the parameters of the dictionary, $\omega_i=\pm 1,i=1,2$ and $b\in [-2,2]$. We present $c=-1$ and $-1e+06$ respectively and Table \ref{table:2D_nega_c1}-\ref{table:2D_nega_c4} show the results. Note that the degrees of freedom (dof) of the shallow neural network are the number of parameters of the dictionary.

\begin{table}[H]
\begin{center}
{{\begin{tabular}{|c|c|c|c|c|c|}
\cline{1-6}
$n$ & dof& $\Vert u-u_h\Vert _{L^2}$ & order & $\Vert u-u_h\Vert _{H^1}$ & order \\ \hline 
\cline{1-6} 
16 & 48 & 4.131e-01 & - & 1.764e+01 & - \\ \hline 
32 & 96 & 2.709e-01 & 0.61 & 1.048e+01 & 0.75 \\ \hline 
64 & 192 & 2.433e-02 & 3.48 & 2.585e+00 & 2.02 \\ \hline 
128 & 384 & 4.591e-03 & 2.41 & 8.878e-01 & 1.54 \\ \hline 
256 & 768 & 7.009e-04 & 2.71 & 2.430e-01 & 1.87 \\ \hline
\end{tabular}
\caption{$L^2$ error, $H^1$ error and order for greedy algorithm solving 2D elliptic equation when $c=-1$}
\label{table:2D_nega_c1}
}}\end{center}
\end{table}
\begin{table}[H]
\begin{center}
{{\begin{tabular}{|c|c|c|c|c|c|}
\cline{1-6}
$n$ & dof& $\Vert u-u_h\Vert _{L^2}$ & order & $\Vert u-u_h\Vert _{H^1}$ & order \\ \hline 
\cline{1-6} 
16 & 48 & 4.245e-01 & - & 1.854e+01 & - \\ \hline 
32 & 96 & 2.787e-01 & 0.61 & 1.163e+01 & 0.67 \\ \hline 
64 & 192 & 2.232e-02 & 3.64 & 2.785e+00 & 2.06 \\ \hline 
128 & 384 & 4.515e-03 & 2.31 & 9.733e-01 & 1.52 \\ \hline 
256 & 768 & 6.633e-04 & 2.77 & 2.827e-01 & 1.78 \\ \hline
\end{tabular}
\caption{$L^2$ error, $H^1$ error and order for greedy algorithm solving 2D elliptic equation when $c=-1e+06$}
\label{table:2D_nega_c4}
}}\end{center}
\end{table}
Next, we present the result of FEM for different $c$ with different $h$. We use the standard triangle Lagrange elements $P_k\ (k=1,2)$ where $k$ denotes the degree of the polynomials. The degrees of freedom (dof) for $P_k$ FEM is $(k/h)^2-1$. The results show that under the same degree of freedom, the OGA method can reach much lower errors. 
\begin{table}[H]
\addtolength{\tabcolsep}{-3.5pt}
\begin{center}
{{\begin{tabular}{|c|c|c|c|c|c|c|c|c|c|c|}
\hline
\multirow{2}{*}{$h$} &  \multicolumn{5}{c|}{$P_1$ linear FEM} & \multicolumn{5}{c|}{$P_2$ quadratic FEM} \\ \cline{2-11}
& dof & $\Vert u-u_h\Vert _{L^2}$ & order & $\Vert u-u_h\Vert _{H^1}$ & order & dof & $\Vert u-u_h\Vert _{L^2}$ & order & $\Vert u-u_h\Vert _{H^1}$ & order \\ \hline 
\cline{1-11} 
1/8 & 63 & 1.375e+00 & - & 2.642e+01 & - & 255 & 4.270e-01 & - & 1.761e+01 & - \\ \hline 
1/16 & 255 & 4.699e-01 & 1.55 & 1.733e+01 & 0.61 & 1023 & 8.864e-02 & 2.27 & 6.693e+00 & 1.40 \\ \hline 
1/32 & 1023 & 1.944e-01 & 1.27 & 1.055e+01 & 0.72 & 4095 & 1.017e-02 & 3.12 & 1.958e+00 & 1.77 \\ \hline 
1/64 & 4095 & 5.846e-02 & 1.73 & 5.432e+00 & 0.96 & 16383 & 1.138e-03 & 3.16 & 5.168e-01 & 1.92 \\ \hline 
1/128 & 16383 & 1.536e-02 & 1.93 & 2.724e+00 & 1.00 & 65535 & 1.365e-04 & 3.06 & 1.312e-01 & 1.98 \\ \hline
1/256 & 65535 & 3.888e-03 & 1.98 & 1.363e+00 & 1.00 & 262143 & 1.686e-05 & 3.02 & 3.293e-02 & 1.99 \\ \hline
\end{tabular}
\caption{$L^2$ error, $H^1$ error and order for FEM solving 2D elliptic equation when $c=-1$}
\label{table:2D_nega_FEM_c1}
}}\end{center}
\end{table}

\begin{table}[H]
\addtolength{\tabcolsep}{-3.5pt}
\begin{center}
{{\begin{tabular}{|c|c|c|c|c|c|c|c|c|c|c|}
\hline
\multirow{2}{*}{$h$} &  \multicolumn{5}{c|}{$P_1$ linear FEM} & \multicolumn{5}{c|}{$P_2$ quadratic FEM} \\ \cline{2-11}
& dof & $\Vert u-u_h\Vert _{L^2}$ & order & $\Vert u-u_h\Vert _{H^1}$ & order & dof & $\Vert u-u_h\Vert _{L^2}$ & order & $\Vert u-u_h\Vert _{H^1}$ & order \\ \hline 
\cline{1-11} 
1/8 & 63 & 5.427e-01 & - & 2.344e+01 & - & 255 & 2.881e-01 & - & 1.902e+01 & - \\ \hline 
1/16 & 255 & 2.924e-01 & 0.89 & 1.813e+01 & 0.37 & 1023 & 5.823e-02 & 2.31 & 7.263e+00 & 1.39 \\ \hline 
1/32 & 1023 & 9.111e-02 & 1.68 & 1.040e+01 & 0.80 & 4095 & 8.304e-03 & 2.81 & 2.040e+00 & 1.83 \\ \hline 
1/64 & 4095 & 2.408e-02 & 1.92 & 5.388e+00 & 0.95 & 16383 & 1.084e-03 & 2.94 & 5.291e-01 & 1.95 \\ \hline 
1/128 & 16383 & 6.104e-03 & 1.98 & 2.718e+00 & 0.99 & 65535 & 1.323e-02 & -3.61 & 1.323e+01 & -4.64 \\ \hline
1/256 & 65535 & 1.599e-03 & 1.93 & 1.438e+00 & 0.92 & 262143 & 3.231e-05 & 8.68 & 4.300e-02 & 8.27 \\ \hline
\end{tabular}
\caption{$L^2$ error, $H^1$ error and order for FEM solving 2D elliptic equation when $c=-1e+06$}
\label{table:2D_nega_FEM_c4}
}}\end{center}
\end{table}

\subsection{Example 3}
We consider the 2D elliptic equation
\begin{equation}
		\left\{
			\begin{aligned}
				 &-\nabla\cdot(A\nabla u)+cu=f,\quad x\in (0,1)^2, \\
				 &B^1(u)=0, \quad x\in \partial(0,1)^2, \\
			\end{aligned}
		\right.
	\end{equation}
with the analytical solution $u(x,y)=\sin^2(2\pi x)\sin^2(2\pi y)\cos^2(2\pi x)\cos^2(2\pi y)$. We apply the orthogonal greedy algorithm (OGA) with dictionary $\mathbb{P}_2^2$ (i.e. corresponding to ${\rm{ReLU}}^2$). For the parameters of the dictionary, $\omega=(\cos(\theta), \sin(\theta))$, $\theta \in [0, 2\pi]$ and $b\in [-2,2]$. We present $A=[2,1;1,3]$ and $c=-1$ and $-1e+06$ respectively and Table \ref{table:exam3_2D_nega_c1}-\ref{table:exam3_2D_nega_c4} show the results.
\begin{table}[H]
\begin{center}
{{\begin{tabular}{|c|c|c|c|c|c|}
\cline{1-6}
$n$ & dof& $\Vert u-u_h\Vert _{L^2}$ & order & $\Vert u-u_h\Vert _{H^1}$ & order \\ \hline 
\cline{1-6} 
16 & 32 & 1.508e-02 & 0.38 & 4.041e-01 & 0.20 \\ \hline 
32 & 64 & 9.989e-03 & 0.59 & 3.086e-01 & 0.39 \\ \hline 
64 & 128 & 3.716e-03 & 1.43 & 1.360e-01 & 1.18 \\ \hline 
128 & 256 & 9.624e-04 & 1.95 & 5.825e-02 & 1.22 \\ \hline 
256 & 512 & 3.262e-04 & 1.56 & 2.889e-02 & 1.01 \\ \hline
\end{tabular}
\caption{$L^2$ error, $H^1$ error and order for greedy algorithm solving 2D elliptic equation when $c=-1$}
\label{table:exam3_2D_nega_c1}
}}\end{center}
\end{table}

\begin{table}[H]
\begin{center}
{{\begin{tabular}{|c|c|c|c|c|c|}
\cline{1-6}
$n$ & dof& $\Vert u-u_h\Vert _{L^2}$ & order & $\Vert u-u_h\Vert _{H^1}$ & order \\ \hline 
\cline{1-6} 
16 & 32 & 1.251e-02 & 0.46 & 3.887e-01 & 0.28 \\ \hline 
32 & 64 & 8.387e-03 & 0.58 & 3.309e-01 & 0.23 \\ \hline 
64 & 128 & 3.822e-03 & 1.13 & 1.765e-01 & 0.91 \\ \hline 
128 & 256 & 1.245e-03 & 1.62 & 8.581e-02 & 1.04 \\ \hline 
256 & 512 & 4.651e-04 & 1.42 & 4.545e-02 & 0.92 \\ \hline
\end{tabular}
\caption{$L^2$ error, $H^1$ error and order for greedy algorithm solving 2D elliptic equation when $c=-1e+06$}
\label{table:exam3_2D_nega_c4}
}}\end{center}
\end{table}

Next, we present the result of FEM for different $c$ with different $h$. We use the standard triangle Lagrange elements $P_k\ (k=1,2)$ where $k$ denotes the degree of the polynomials. The degrees of freedom (dof) for $P_k$ FEM is $(k/h)^2-1$. The results show that under the same degree of freedom, the OGA method can reach much lower errors. 
\begin{table}[H]
\addtolength{\tabcolsep}{-3.5pt}
\begin{center}
{{\begin{tabular}{|c|c|c|c|c|c|c|c|c|c|c|}
\hline
\multirow{2}{*}{$h$} &  \multicolumn{5}{c|}{$P_1$ linear FEM} & \multicolumn{5}{c|}{$P_2$ quadratic FEM} \\ \cline{2-11}
& dof & $\Vert u-u_h\Vert _{L^2}$ & order & $\Vert u-u_h\Vert _{H^1}$ & order & dof & $\Vert u-u_h\Vert _{L^2}$ & order & $\Vert u-u_h\Vert _{H^1}$ & order \\ \hline 
\cline{1-11} 
1/8 & 63 & 2.241e-02 & - & 4.398e-01 & - & 255 & 6.679e-03 & - & 2.428e-01 & - \\ \hline 
1/16 & 255 & 8.592e-03 & 1.38 & 2.624e-01 & 0.75 & 1023 & 7.425e-04 & 3.17 & 6.750e-02 & 1.85 \\ \hline 
1/32 & 1023 & 2.626e-03 & 1.71 & 1.387e-01 & 0.92 & 4095 & 8.576e-05 & 3.11 & 1.846e-02 & 1.87 \\ \hline 
1/64 & 4095 & 6.971e-04 & 1.91 & 7.015e-02 & 0.98 & 16383 & 1.033e-05 & 3.05 & 4.741e-03 & 1.96 \\ \hline 
1/128 & 16383 & 1.770e-04 & 1.98 & 3.516e-02 & 1.00 & 65535 & 1.277e-06 & 3.02 & 1.194e-03 & 1.99 \\ \hline
1/256 & 65535 & 4.443e-05 & 1.99 & 1.759e-02 & 1.00 & 262143 & 1.592e-07 & 3.00 & 2.990e-04 & 2.00 \\ \hline
\end{tabular}
\caption{$L^2$ error, $H^1$ error and order for FEM solving 2D elliptic equation when $c=-1$}
\label{table:2D_exam3_nega_FEM_c1}
}}\end{center}
\end{table}

\begin{table}[H]
\addtolength{\tabcolsep}{-3.5pt}
\begin{center}
{{\begin{tabular}{|c|c|c|c|c|c|c|c|c|c|c|}
\hline
\multirow{2}{*}{$h$} &  \multicolumn{5}{c|}{$P_1$ linear FEM} & \multicolumn{5}{c|}{$P_2$ quadratic FEM} \\ \cline{2-11}
& dof & $\Vert u-u_h\Vert _{L^2}$ & order & $\Vert u-u_h\Vert _{H^1}$ & order & dof & $\Vert u-u_h\Vert _{L^2}$ & order & $\Vert u-u_h\Vert _{H^1}$ & order \\ \hline 
\cline{1-11} 
1/8 & 63 & 7.345e-03 & - & 3.282e-01 & - & 255 & 4.370e-03 & - & 2.175e-01 & - \\ \hline 
1/16 & 255 & 4.651e-03 & 0.66 & 2.582e-01 & 0.35 & 1023 & 5.984e-04 & 2.87 & 7.063e-02 & 1.62 \\ \hline
1/32 & 1023 & 1.275e-03 & 1.87 & 1.377e-01 & 0.91 & 4095 & 8.018e-05 & 2.90 & 1.883e-02 & 1.91 \\ \hline 
1/64 & 4095 & 3.264e-04 & 1.97 & 7.000e-02 & 0.98 & 16383 & 1.043e-05 & 2.94 & 4.900e-03 & 1.94 \\ \hline 
1/128 & 16383 & 8.210e-05 & 2.00 & 3.514e-02 & 0.99 & 65535 & 1.499e-06 & 2.80 & 1.309e-03 & 1.90 \\ \hline
1/256 & 65535 & 2.058e-05 & 2.00 & 1.760e-02 & 1.00 & 262143 & 1.640e-07 & 3.19 & 3.005e-04 & 2.12 \\ \hline
\end{tabular}
\caption{$L^2$ error, $H^1$ error and order for FEM solving 2D elliptic equation when $c=-1e+06$}
\label{table:2D_exam3_nega_FEM_c4}
}}\end{center}
\end{table}

\subsection{Example 4}
We consider the 3D elliptic equation
\begin{equation}
		\left\{
			\begin{aligned}
				 &-\Delta u+cu=f,\quad x\in (0,1)^3, \\
				 &\frac{\partial u}{\partial n}=0, \quad x\in \partial(0,1)^3,
			\end{aligned}
		\right.
	\end{equation}
with the analytical solution $u(x,y)=\cos(2\pi x)\cos(2\pi y)\cos(2\pi z)$. We apply the orthogonal greedy algorithm (OGA) with dictionary $\mathbb{P}_2^3$ (i.e. corresponding to ${\rm{ReLU}}^2$). For the parameters of the dictionary, $\omega_i=\pm 1,i=1,2,3$ and $b\in [-2,2]$. We present $c=-1$ and $-1e+06$ respectively and Table \ref{table:3D_nega_c1}-\ref{table:3D_nega_c2} show the results.
\begin{table}[H]
\begin{center}
{{\begin{tabular}{|c|c|c|c|c|c|}
\cline{1-6}
$n$ & dof& $\Vert u-u_h\Vert _{L^2}$ & order & $\Vert u-u_h\Vert _{H^1}$ & order \\ \hline 
\cline{1-6} 
16 & 64 & 2.530e-01 & - & 2.713e+00 & - \\ \hline 
32 & 128 & 2.699e-02 & 3.23 & 5.899e-01 & 2.20 \\ \hline 
64 & 256 & 5.618e-03 & 2.26 & 2.132e-01 & 1.47 \\ \hline 
128 & 512 & 6.093e-04 & 3.20 & 4.460e-02 & 2.26 \\ \hline 
256 & 1024 & 1.099e-04 & 2.47 & 9.620e-03 & 2.21 \\ \hline
\end{tabular}
\caption{$L^2$ error, $H^1$ error and order for greedy algorithm solving 3D elliptic equation when $c=-1$}
\label{table:3D_nega_c1}
}}\end{center}
\end{table}
\begin{table}[H]
\begin{center}
{{\begin{tabular}{|c|c|c|c|c|c|}
\cline{1-6}
$n$ & dof& $\Vert u-u_h\Vert _{L^2}$ & order & $\Vert u-u_h\Vert _{H^1}$ & order \\ \hline 
\cline{1-6} 
16 & 64 & 2.526e-01 & - & 2.838e+00 & - \\ \hline 
32 & 128 & 2.943e-02 & 3.10 & 6.994e-01 & 2.02 \\ \hline 
64 & 256 & 5.562e-03 & 2.40 & 2.376e-01 & 1.56 \\ \hline 
128 & 512 & 5.023e-04 & 3.47 & 4.736e-02 & 2.33 \\ \hline 
256 & 1024 & 5.746e-05 & 3.13 & 1.031e-02 & 2.20 \\ \hline
\end{tabular}
\caption{$L^2$ error, $H^1$ error and order for greedy algorithm solving 3D elliptic equation when $c=-1e+06$}
\label{table:3D_nega_c2}
}}\end{center}
\end{table}
Next, we present the result of FEM for different $c$ with different $h$. We use the standard triangle Lagrange elements $P_k\ (k=1,2)$ where $k$ denotes the degree of the polynomials. The degrees of freedom (dof) for $P_k$ FEM is $(k/h)^3-1$. The results show that under the same degree of freedom, the OGA method can reach much lower errors. 
\begin{table}[H]
\addtolength{\tabcolsep}{-3.5pt}
\begin{center}
{{\begin{tabular}{|c|c|c|c|c|c|c|c|c|c|c|}
\hline
\multirow{2}{*}{$h$} &  \multicolumn{5}{c|}{$P_1$ linear FEM} & \multicolumn{5}{c|}{$P_2$ quadratic FEM} \\ \cline{2-11}
& dof & $\Vert u-u_h\Vert _{L^2}$ & order & $\Vert u-u_h\Vert _{H^1}$ & order & dof & $\Vert u-u_h\Vert _{L^2}$ & order & $\Vert u-u_h\Vert _{H^1}$ & order \\ \hline 
\cline{1-11} 
1/4 & 63 & 2.682e-01 & - & 3.119e+00 & - & 511 & 6.002e-02 & - & 1.090e+00 & - \\ \hline 
1/8 & 511 & 1.263e-01 & 1.09 & 1.859e+00 & 0.75 & 4095 & 6.949e-03 & 3.11 & 3.249e-01 & 1.75 \\ \hline 
1/16 & 4095 & 3.984e-02 & 1.66 & 9.655e-01 & 0.95 & 32767 & 7.547e-04 & 3.20 & 8.822e-02 & 1.88 \\ \hline 
1/32 & 32767 & 1.063e-02 & 1.91 & 4.864e-01 & 0.99 & 262143 & 8.904e-05 & 3.08 & 2.273e-02 & 1.96 \\ \hline 
1/64 & 262143 & 2.703e-03 & 1.98 & 2.437e-01 & 1.00 & 2097151 & 1.097e-05 & 3.02 & 5.743e-03 & 1.98 \\ \hline 
\end{tabular}
\caption{$L^2$ error, $H^1$ error and order for FEM solving 3D elliptic equation when $c=-1$}
\label{table:3D_exam3_nega_FEM_c1}
}}\end{center}
\end{table}

\begin{table}[H]
\addtolength{\tabcolsep}{-3.5pt}
\begin{center}
{{\begin{tabular}{|c|c|c|c|c|c|c|c|c|c|c|}
\hline
\multirow{2}{*}{$h$} &  \multicolumn{5}{c|}{$P_1$ linear FEM} & \multicolumn{5}{c|}{$P_2$ quadratic FEM} \\ \cline{2-11}
& dof & $\Vert u-u_h\Vert _{L^2}$ & order & $\Vert u-u_h\Vert _{H^1}$ & order & dof & $\Vert u-u_h\Vert _{L^2}$ & order & $\Vert u-u_h\Vert _{H^1}$ & order \\ \hline 
\cline{1-11} 
1/4 & 63 & 1.956e-01 & - & 3.267e+00 & - & 511 & 3.850e-02 & - & 1.252e+00 & - \\ \hline 
1/8 & 511 & 6.314e-02 & 1.63 & 1.863e+00 & 0.81 & 4095 & 5.410e-03 & 2.83 & 3.549e-01 & 1.82 \\ \hline 
1/16 & 4095 & 1.686e-02 & 1.90 & 9.641e-01 & 0.95 & 32767 & 6.964e-04 & 2.96 & 9.168e-02 & 1.95 \\ \hline 
1/32 & 32767 & 4.286e-03 & 1.98 & 4.863e-01 & 0.99 & 262143 & 8.779e-05 & 2.99 & 2.314e-02 & 1.99 \\ \hline 
1/64 & 262143 & 1.076e-03 & 1.99 & 2.437e-01 & 1.00 & 2097151 & 1.170e-05 & 2.91 & 6.961e-03 & 1.73 \\ \hline 
\end{tabular}
\caption{$L^2$ error, $H^1$ error and order for FEM solving 3D elliptic equation when $c=-1e+06$}
\label{table:3D_exam3_nega_FEM_c2}
}}\end{center}
\end{table}

\subsection{Example 5: Helmholtz Equation }

We consider the 2D Helmholtz equation given as follows: 
\begin{equation}
		\left\{
			\begin{aligned}
				& -\Delta u-k^2u=f,\quad x\in (0,1)^2, \\
				& \frac{\partial u}{\partial n}=0,\quad  x\in \partial(0,1)^2,
			\end{aligned}
		\right.
	\end{equation} 
where $k>1$ is the wavenumber. When applied to a large wavenumber and the solution is related to $k$, the problem will be difficult to solve due to the high indefiniteness. In this example, the source term $f=k^2\cos(k x)\cos(k y)-k^2$ and the analytical solution is $u(x,y)=\cos(k x)\cos(k y)+1$. We tested $k=2\pi$ and $10\pi$ both on the greedy algorithm for the ${\rm{ReLU}}^3$ activation function and on the FEM for the $P_2$ quadratic elements. The results are shown in Tables \ref{table:2D_exam5_Greedy}-\ref{table:2D_exam5_FEM}, which show that greedy algorithm can reach much lower error than FEM under the same dof.

\begin{table}[H]
\addtolength{\tabcolsep}{-3.5pt}
\begin{center}
{{\begin{tabular}{|c|c|c|c|c|c|c|c|c|c|c|}
\hline
\multirow{2}{*}{$n$} &  \multicolumn{5}{c|}{$k=2\pi$} & \multicolumn{5}{c|}{$k=10\pi$} \\ \cline{2-11}
& dof & $\Vert u-u_h\Vert _{L^2}$ & order & $\Vert u-u_h\Vert _{H^1}$ & order & dof & $\Vert u-u_h\Vert _{L^2}$ & order & $\Vert u-u_h\Vert _{H^1}$ & order \\ \hline 
\cline{1-11} 
16 & 48 & 2.801e-02 & - & 1.382e-01 & - & 48 & 2.150e+00 & - & 3.399e+00 & - \\ \hline 
32 & 96 & 2.229e-03 & 3.65 & 2.600e-02 & 2.41 & 96 & 7.305e-01 & 1.56 & 1.193e+00 & 1.51 \\ \hline 
64 & 192 & 2.952e-04 & 2.92 & 6.435e-03 & 2.01 & 192 & 6.410e-02 & 3.51 & 1.719e-01 & 2.80 \\ \hline 
128 & 384 & 3.887e-05 & 2.93 & 1.649e-03 & 1.96 & 384 & 5.674e-03 & 3.50 & 4.195e-02 & 2.03 \\ \hline 
256 & 768 & 1.176e-05 & 1.73 & 4.430e-04 & 1.90 & 768 & 8.196e-04 & 2.79 & 1.134e-02 & 1.89 \\ \hline 
\end{tabular}
\caption{$L^2$ error, $H^1$ error and order for greedy algorithm solving 2D Helmholtz equation}
\label{table:2D_exam5_Greedy}
}}\end{center}
\end{table}

\begin{table}[H]
\addtolength{\tabcolsep}{-3.5pt}
\begin{center}
{{\begin{tabular}{|c|c|c|c|c|c|c|c|c|c|c|}
\hline
\multirow{2}{*}{$h$} &  \multicolumn{5}{c|}{$k=2\pi$} & \multicolumn{5}{c|}{$k=10\pi$} \\ \cline{2-11}
& dof & $\Vert u-u_h\Vert _{L^2}$ & order & $\Vert u-u_h\Vert _{H^1}$ & order & dof & $\Vert u-u_h\Vert _{L^2}$ & order & $\Vert u-u_h\Vert _{H^1}$ & order \\ \hline 
\cline{1-11} 
1/16 & 1023 & 7.957e-02 & - & 1.267e-01 & - & 1023 & 1.419e+00 & - & 2.261e+00 & - \\ \hline
1/32 & 4095 & 4.084e-02 & 0.96 & 6.469e-02 & 0.97 & 4095 & 1.906e-02 & - & 9.205e-02 & - \\ \hline 
1/64 & 16383 & 2.069e-02 & 0.98 & 3.273e-02 & 0.98 & 16383 & 2.073e-02 & - & 4.017e-02 & - \\ \hline 
1/128 & 65535 & 1.041e-02 & 0.99 & 1.647e-02 & 0.99 & 65535 & 1.041e-02 & 0.99 & 1.749e-02 & 1.20 \\ \hline
1/256 & 262143 & 5.224e-03 & 0.99 & 8.260e-03 & 1.00 & 262143 & 5.224e-03 & 0.99 & 8.391e-03 & 1.06 \\ \hline
1/512 & 1048575 & 2.616e-03 & 1.00 & 4.136e-03 & 1.00 & 1048575 & 2.616e-03 & 1.00 & 4.153e-03 & 1.01 \\ \hline
\end{tabular}
\caption{$L^2$ error, $H^1$ error and order for FEM solving 2D Helmholtz equation}
\label{table:2D_exam5_FEM}
}}\end{center}
\end{table}

\section{Conclusion}
We present a priori error estimate for the shallow neural network for the indefinite elliptic partial differential equation. The OGA method is then shown to produce the solution that satisfies the theoretical accuracy rate. The errors of the method consist of modeling errors and discretization errors. We analyze the priori error estimate and convergence analysis of the different components. Such an error analysis is novel in the framework of shallow neural network theory. We also present several numerical examples in comparions with those of Finite element methods. The results show that the shallow neural network can achieve better accuracy than those from FEM when the degrees of freedom are comparable. In our future work, we shall establish the convergence analysis of OGA which is justfied only numerically in this paper.

\bibliographystyle{plain}
\bibliography{bibfile}

\end{document}